\documentclass{amsart}
\usepackage{amsmath}
\usepackage{amsfonts}
\usepackage{graphicx}

\newtheorem{theorem}{Theorem}[section]
\newtheorem*{theorem A}{Theorem A}
\newtheorem*{theorem B}{N\"olker's Theorem}
\newtheorem{lemma}{Lemma}[section]

\newtheorem{corollary}{Corollary}[section]
\newtheorem{problem}{Problem}

\theoremstyle{remark}

\theoremstyle{remark}

\theoremstyle{definition}

\numberwithin{equation}{section}
\def\({\left ( }
\def\){\right )}
\def\<{\left < }
\def\>{\right >}

\input{tcilatex}

\begin{document}
\title{affine translation hypersurfaces in Euclidean and isotropic spaces}
\author{Muhittin Evren Aydin}
\address{Department of Mathematics, Faculty of Science, Firat University,
Elazig, 23200, Turkey}
\email{meaydin@firat.edu.tr}
\thanks{}

\begin{abstract}
In this paper, we extend the notion of affine translation surfaces
introduced by Liu and Yu (Proc. Japan Acad. Ser. A Math. Sci. 89, 111--113,
2013) in a Euclidean space $\mathbb{R}^{3}$ to higher dimensional ambient
spaces. We provide that an affine translation hypersurface of constant
Gauss-Kronocker curvature $K_{0}$ in $\mathbb{R}^{n+1}$ is a cylinder, i.e. $%
K_{0}=0$. As further applications we describe such hypersurfaces in the
isotropic spaces satisfying certain conditions on the isotropic curvatures
and the Laplacian.
\end{abstract}

\maketitle

\section{Introduction}

Let $\mathbb{R}^{n+1}$ be a Euclidean space and $\left(
x_{1},x_{2},...,x_{n+1}\right) $ the orthogonal coordinate system in $%
\mathbb{R}^{n+1}$. Then a hypersurface in $\mathbb{R}^{n+1},$ $n\geq 2,$ is
called \textit{translation hypersurface} if it is the graph of the form%
\begin{equation}
x_{n+1}\left( x_{1},x_{2},...,x_{n}\right) =f_{1}\left( x_{1}\right)
+f_{2}\left( x_{2}\right) +...+f_{n}\left( x_{n}\right) ,  \tag{1.1}
\end{equation}%
where $f_{1},f_{2},...,f_{n}$ are real-valued smooth functions of one
variable (see \cite{2,7,30}). These hypersurfaces are obtained by
translating the curves (called \textit{generating curves}) lying in mutually
orthogonal planes of $\mathbb{R}^{n+1}$.

Dillen et al. \cite{7} proved that a minimal (vanishing mean curvature)
translation hypersurface in $\mathbb{R}^{n+1}$ is either a hyperplane or a
product manifold $M^{2}\times \mathbb{R}^{n-2},$ where $M^{2}$ is \textit{Scherk's
minimal translation surface} in $\mathbb{R}^{3}$ given in explicit form 
\begin{equation*}
x_{3}\left( x_{1},x_{2}\right) =\frac{1}{c}\log \left\vert \frac{\cos \left(
cx_{1}\right) }{\cos \left( cx_{2}\right) }\right\vert ,\text{ }c\in \mathbb{%
R-}\left\{ 0\right\} .
\end{equation*}%
In 3-dimensional context, many different generalizations of Scherk's surface were treated on $\mathbb{A}^{3}$ \cite{9,31}, $Nil_{3}$ \cite{12}, $\mathbb{H}^{3}$ \cite{16}, $Sol_{3}$ \cite{17}, $\mathbb{R}^{3}$ \cite{18,19}.

Constant Gauss-Kronocker curvature (CGKC) and constant mean curvature (CMC)
translation hypersurfaces in $\mathbb{R}^{n+1}$ (also in the
Lorentz-Minkowski space $\mathbb{R}_{1}^{n+1}$) were described in \cite{28}
by Seo. For lightlike counterparts of such results see \cite{11}.

Most recently, Moruz and Munteanu \cite{22} conjectured a new class of
translation hypersurfaces in $\mathbb{R}^{4}$ as the graph of the form%
\begin{equation*}
x_{4}\left( x_{1},x_{2},x_{3}\right) =f_{1}\left( x_{1}\right) +f_{2}\left(
x_{2},x_{3}\right) .
\end{equation*}%
This one appears as the sum of a curve in $x_{1}x_{4}-$plane and a graph surface
in $x_{2}x_{3}x_{4}-$space. Immediately afterwards this new concept was generalized to higher dimensionals by Munteanu et al. \cite{23} as considering the form 
\begin{equation}
x_{n+m+1}\left( x_{1},x_{2},...,x_{n+m}\right) =f_{1}\left(
x_{1},x_{2},...,x_{n}\right) +f_{2}\left( x_{n+1},x_{n+2},...,x_{n+m}\right) \tag{1.2}
.
\end{equation}%
The graph of the form (1.2) in $\mathbb{R}^{n+m+1}$ is called \textit{translation graph}. The authors in \cite{22,23} obtained new classifications and results by imposing the minimality condition. Due to the above framework, the following problems can be stated:

\begin{problem}
To obtain CMC and CGKC translation hypersurfaces in $\mathbb{R}^{n+1}$ (as defined by Dillen
et al.) whose either
\begin{enumerate}
\item the generating curves are planar lying in non-orthogonal
planes; or

\item some of them generating curves are planar, others are not; or

\item the generating curves are all non-planar (space curves).
\end{enumerate}
\end{problem}

\begin{problem}
 To characterize CGKC and CMC translation graphs in $\mathbb{R}^{n+1}$ (as defined by Moruz et al.) without imposing restrictions.
\end{problem}

This study aims to solve a part of first item of Problem 1, that is, to classify
the CGKC translation hypersurfaces whose the generating curves lie in
non-orthogonal planes. For this, we are motivated by the notion of \textit{%
affine translation surface} introduced by Liu and Yu \cite{14} as a graph of
the form 
\begin{equation}
x_{3}\left( x_{1},x_{2}\right) =f_{1}\left( x_{1}\right) +f_{2}\left(
x_{2}+cx_{1}\right)  \notag
\end{equation}%
for some nonzero constant $c$. Such surfaces with CMC were classified in 
\cite{15}. By a change of parameter, its parameterization turns to%
\begin{equation*}
r\left( u,v\right) =\left( u,v-cu,f_{1}\left( u\right) +f_{2}\left( v\right)
\right) ,
\end{equation*}%
which implies that the generating curves lie in non-orthogonal planes. In
order to achieve our purpose, we consider the graph in $\mathbb{R}^{n+1}$ of
the form 
\begin{equation}
x_{n+1}\left( x_{1},x_{2},...,x_{n}\right) =f_{1}\left( y_{1}\right)
+f_{2}\left( y_{2}\right) +...+f_{n}\left( y_{n}\right) ,  \tag{1.3}
\end{equation}%
where 
\begin{equation}
y_{i}=\sum_{j=1}^{n}a_{ij}x_{j},\text{ }i=1,2,...,n.  \tag{1.4}
\end{equation}%
If $A=\left( a_{ij}\right) $ in (1.4) is non-orthogonal regular matrix, then
we call the graph of the form (1.3) \textit{affine translation hypersurface }%
and $\left( y_{1},y_{2},...,y_{n}\right) $ \textit{affine parameter
coordinates.} Note that the generating curves of an affine translaiton hypersurface lie in non-orthgonal planes due to the non-orthogonality of $A$.

In the particular case $y_{1}=x_{1},$ $y_{2}=x_{2},$ ..., $y_{n-1}=x_{n-1}$, Yang and Fu \cite{31} proposed to obtain some curvature classifications
for such a hypersurface in $\mathbb{R}^{n+1}$. In more general case, we provide the following: 

\begin{theorem}
Let $M^{n}$ be an affine translation hypersurface in $\mathbb{R}^{n+1}$ with
CGKC $K_{0}.$ Then it is congruent to a cylinder, i.e. $K_{0}=0$.
\end{theorem}

Combining this with the result of Seo \cite[Theorem 2.5]{28}, we derive:

\begin{corollary}
There is no a translation hypersurface in $\mathbb{R}^{n+1}$ with nonzero
CGKC provided the generating curves are all planar.
\end{corollary}

Further we classify these hypersurfaces in isotropic spaces satisfying certain conditions on the
isotropic curvatures and the Laplacian.

\section{Preliminaries}

\subsection{Basics on hypersurfaces in $\mathbb{R}^{n+1}$}

Let $M^{n},\mathbb{S}^{n},\left\langle \cdot ,\cdot \right\rangle$ and $%
\left\Vert \cdot \right\Vert$ denote a hypersurface, the standard
hypersphere, the Euclidean scalar product and the induced norm of $\mathbb{R}%
^{n+1}$, respectively. For further properties of submanifolds in $\mathbb{R}%
^{n+1}$ see \cite{3}.

The map $\nu :M^{n}\longrightarrow \mathbb{S}^{n}$ in $\mathbb{R}^{n+1}$ is
called \textit{Gauss map} of $M^{n}$ and its differential $d\nu$ is known as
the \textit{shape operator} $A$ of $M^{n}.$ Let $T_{p}M^{n}$ be the tangent
space at a point $p\in M^{n},$ then the following occurs: 
\begin{equation*}
\left\langle A_{p}\left( x_{p}\right) ,y_{p}\right\rangle =\left\langle d\nu
\left( x_{p}\right) ,y_{p}\right\rangle ,\text{ }x_{p},y_{p}\in T_{p}M^{n},
\end{equation*}%
where the induced metric on $M^{n}$ from $\mathbb{R}^{n+1}$ is denoted by
same symbol $\left\langle \cdot ,\cdot \right\rangle .$

The real number $\det \left( A_{p}\right) $ is called the \textit{%
Gauss-Kronocker curvature }of $M^{n}$ at $p\in M^{n}$. A hypersurface in $%
\mathbb{R}^{n+1}$ for which the Gauss-Kronocker curvature at each point is
zero is called \textit{flat.}

The graph hypersurface in $\mathbb{R}^{n+1}$ of a given real-valued smooth
function $z=z\left( x_{1},x_{2},...,x_{n}\right) $ is of the form%
\begin{equation*}
r:\mathbb{R}^{n}\longrightarrow \mathbb{R}^{n+1}, \text{ } r\left(
x_{1},x_{2},...,x_{n}\right) =\left( x_{1},x_{2},...,x_{n},z\left(
x_{1},x_{2},...,x_{n}\right) \right) .
\end{equation*}%
The Gauss-Kronecker curvature $K$ of such a hypersurface in $\mathbb{R}%
^{n+1} $ turns to%
\begin{equation}
K=\frac{\det \left( Hess\left( z\right) \right) }{\left(
1+\sum_{i=1}^{n}\left( z_{,x_{i}}\right) ^{2}\right) ^{\frac{n+2}{2}}}, 
\tag{2.1}
\end{equation}%
where $z_{,x_{i}}=\frac{\partial z}{\partial x_{i}}$ and $Hess\left(
z\right) $ is the Hessian of $z$, namely%
\begin{equation}
Hess\left( z\right) =%
\begin{bmatrix}
z_{,x_{1}x_{1}} & z_{,x_{1}x_{2}} & ... & z_{,x_{1}x_{n}} \\ 
z_{,x_{2}x_{1}} & z_{,x_{2}x_{2}} & ... & z_{,x_{2}x_{n}} \\ 
\vdots & \vdots & ... & \vdots \\ 
z_{,x_{n}x_{1}} & z_{,x_{n}x_{2}} & ... & z_{,x_{n}x_{n}}%
\end{bmatrix}
\tag{2.2}
\end{equation}
for $z_{,x_{i}x_{j}}=\frac{\partial ^{2}z}{\partial x_{i}\partial x_{j}},$ $%
i,j=1,2,...,n$.

\subsection{Basics on hypersurfaces in $\mathbb{I}^{n+1}$}

For general references of the isotropic space $\mathbb{I}^{n+1}$ we refer to 
\cite{5,8,20,21} and \cite{24}-\cite{27}. $\mathbb{I}^{n+1}$ is based on the
following group of motions%
\begin{equation}
\begin{bmatrix}
A & 0 \\ 
B & 1%
\end{bmatrix}%
,  \tag{2.3}
\end{equation}%
where $A \in \mathbb{R}_{n}^{n}$ is an orthonogal $n\times n-$matrix and $%
B\in \mathbb{R}_{1}^{n}$ is a $\left( 1\times n\right) -$matrix.

The \textit{isotropic distance} of $\mathbb{I}^{n+1}$\ which is an invariant
under $\left( 2.3\right) $ is defined as 
\begin{equation}
\left\Vert p-q\right\Vert _{i}=\sqrt{\sum_{j=1}^{n}\left( q_{j}-p_{j}\right)
^{2}}  \tag{2.4}
\end{equation}%
for $p=\left( p_{1},p_{2},...,p_{n+1}\right) ,$ $q=\left(
q_{1},q_{2},...,q_{n+1}\right) \in \mathbb{I}^{n+1}.$ Thereby $\mathbb{I}%
^{n+1}$ can appear as a real affine space endowed with the metric (2.4).

Let $\left( x_{1}, x_{2},...,x_{n+1}\right) $ be the standart affine
coordinates of $\mathbb{I}^{n+1}.$ The metric (2.4) is degenerate along $%
x_{n+1}-$direction and we call the lines in $x_{n+1}-$direction \textit{%
isotropic} \textit{lines}. The $k-$plane involving an isotropic line is
called \textit{isotropic }$k-$\textit{plane.}

A hypersurface in $\mathbb{I}^{n+1}$ is called \textit{admissible} if
nowhere it has isotropic tangent hyperplane.

A \textit{graph hypersurface} $M^{n}$ in $\mathbb{I}^{n+1}$ of a given
smooth function $z\left( x_{1},x_{2},...,x_{n}\right)$ is of the form%
\begin{equation*}
r:\mathbb{R}^{n}\longrightarrow \mathbb{I}^{n+1}, \text{ } r\left(
x_{1},x_{2},...,x_{n}\right) =\left( x_{1},x_{2},...,x_{n},z\left(
x_{1},x_{2},...,x_{n}\right) \right) .
\end{equation*}
Note that $M^{n}$ is admissible since its tangent hyperplane spanned by $%
\left\{ r_{,x_{1}},r_{,x_{2}},...,r_{,x_{n}}\right\} $ does not involve an
isotropic line.

The induced metric $\left\langle \cdot ,\cdot \right\rangle $ on $M^{n}$
from $\mathbb{I}^{n+1}$ is given by 
\begin{equation}
\left\langle \cdot ,\cdot \right\rangle =dx_{1}^{2}+...+dx_{n}^{2}. 
\tag{2.5}
\end{equation}%
Thus, its Laplacian becomes%
\begin{equation}
\bigtriangleup =\sum_{i=1}^{n}\frac{\partial ^{2}}{\partial x_{i}^{2}}. 
\tag{2.6}
\end{equation}

Now let us consider a curve on $M^{n}$ that has the position vector 
\begin{equation}
r=r\left( s\right) =\mathbf{x}\left( s\right) +z\left( s\right) e_{n+1}, 
\tag{2.7}
\end{equation}
where 
\begin{equation}
\mathbf{x}\left( s\right) =\left( x_{1}\left( s\right) ,x_{2}\left( s\right)
,...,x_{n}\left( s\right),0 \right), \text{ } e_{n+1}=\left( \underset{%
n-tuple}{\underbrace{0,0,...0}},1\right) .  \notag
\end{equation}
Derivating of (2.7) with respect to $s$ leads to%
\begin{equation}
r^{\prime } =\mathbf{x}^{\prime }+ \left\langle \mathbf{x}^{\prime } ,\nabla
z \right\rangle e_{n+1},  \tag{2.8}
\end{equation}%
where $\nabla $ denotes the gradient operator in $\mathbb{R}^{n}.$ By again
derivating of $\left( 2.8\right) $ with respect to $s,$ we arrange the
following%
\begin{equation}
r^{\prime \prime } =\mathbf{x}^{\prime \prime }+ \left\langle \mathbf{x}%
^{\prime \prime } ,\nabla z \right\rangle e_{n+1}+ \left( \mathbf{X}^{\prime
}\right)^{T}\cdot Hess\left( z \right)\cdot \mathbf{X}^{\prime } e_{n+1}, 
\tag{2.9}
\end{equation}%
where $\mathbf{X}^{\prime } $ is column matrix associated to $\mathbf{x}%
^{\prime }$ and $\left( \mathbf{X}^{\prime }\right)^{T}$ its transpose.
Therefore, in (2.9), the following decomposition occurs:%
\begin{equation*}
Tan\left( r^{\prime \prime } \right) =\mathbf{x}^{\prime \prime }
+\left\langle \mathbf{x}^{\prime \prime } ,\nabla z \right\rangle e_{n+1}
\end{equation*}%
and 
\begin{equation*}
Nor\left( r^{\prime \prime } \right) = \left( \mathbf{X}^{\prime
}\right)^{T}\cdot Hess\left( z \right) \cdot \mathbf{X}^{\prime } e_{n+1},
\end{equation*}%
where $Tan\left( r^{\prime \prime }\right)$ implies the projection of $%
r^{\prime \prime }$ onto tangent hyperplane of $M^{n}$ and $Nor\left(
r^{\prime \prime } \right)$ the isotropic component of $r^{\prime \prime }$
which is normal to $M^{n}$.

If $\left\Vert Tan\left( r^{\prime \prime } \right) \right\Vert _{i} \neq 0$
then it is called \textit{geodesic curvature function} $\kappa _{G}$ of $r .$
Otherwise $\kappa _{G}=1$ is assumed. Accordingly the following function is
called \textit{normal curvature function} $\kappa _{N}$ of $r$: 
\begin{equation}
\kappa _{N}= \left( \mathbf{X}^{\prime }\right)^{T}\cdot Hess\left( z
\right)\cdot \mathbf{X}^{\prime }.  \tag{2.10}
\end{equation}

The extremal values $\kappa _{1},...,\kappa _{n}$ of (2.10) corresponding to
the eigenvalue functions of $Hess\left( z\right) $ are called \textit{%
principal curvatures} of $M^{n}.$ Since $Hess\left( z \right) $ is
symmetric, all eigenvalue functions are real. Thus one gives rise to define
the following certain curvature functions:%
\begin{equation}
K_{i}=\frac{1}{\binom{n}{i}}\left( \kappa _{1}...\kappa _{i}+\kappa
_{1}...\kappa _{i-1}\kappa _{i+1}+...+\kappa _{n-i+1}...\kappa _{n}\right) .
\tag{2.11}
\end{equation}%
By $\left( 2.11\right) ,$ the \textit{isotropic mean curvature function} $%
H=K_{1}$ is%
\begin{equation}
H=\frac{1}{n}trace\left( Hess\left( z\right) \right) =\frac{1}{n}%
\bigtriangleup z  \tag{2.12}
\end{equation}%
and the \textit{relative curvature} (or \textit{isotropic Gaussian curvature}%
) \textit{function } $K=K_{n}$ 
\begin{equation}
K=\det \left( Hess\left( z\right) \right).  \tag{2.13}
\end{equation}

A hypersurface in $\mathbb{I}^{n+1}$ with vanishing relative curvature (resp.
isotropic mean curvature) is called \textit{isotropic flat }(resp. \textit{%
isotropic minimal}).

\section{Affine translation hypersurfaces in\textbf{\ }$\mathbb{R}^{n+1}$}

Let $x=\left( x_{1},x_{2},...,x_{n}\right) $ denote the orthogonal
coordinate system in $\mathbb{R}^{n}$ and \newline
$z:\mathbb{R}^{n}\longrightarrow \mathbb{R}$, $z=z\left( y\right) ,$ be a
smooth function, where 
\begin{equation}
y=\left( y_{1},y_{2},...,y_{n}\right) ,\text{ }y_{i}=%
\sum_{j=1}^{n}a_{ij}x_{j},\text{ }a_{ij}\in \mathbb{R},\text{ }i=1,2,...,n. 
\tag{3.1}
\end{equation}%
If $A=\left( a_{ij}\right) $ is a non-orthogonal $n\times n-$matrix and $%
\det \left( A\right) \neq 0,$ then we call the graph of $z\left( y\right) $
in $\mathbb{R}^{n+1}$ \textit{affine graph} of $z\left( x\right) $ and $%
\left( y_{1},y_{2},...,y_{n}\right) $ \textit{affine parameter coordinates}.

Hence we provide the following result to use later.

\begin{lemma}
Let $z\left( y\right) $ be a smooth real-valued function on $\mathbb{R}^{n}$%
, where $y$ is the affine parameter coordinates given by $\left( 3.1\right)
. $ Then the following relation holds:%
\begin{equation}
\det \left[ Hess\left( z\left( x\right) \right) \right] =\det \left[ A\right]
^{2}\det \left[ Hess\left( z\left( y\right) \right) \right]  \tag{3.2}
\end{equation}%
for $x=\left( x_{1},x_{2},...,x_{n}\right) .$
\end{lemma}

\begin{proof}
The partial derivatives of $z$ with respect to $x_{i},$ $1\leq i\leq n$ ,
gives%
\begin{equation*}
z_{,x_{i}}=\sum_{k=1}^{n}a_{ki}z_{,y_{k}},\text{ }z_{,x_{i}x_{j}}=%
\sum_{k,l=1}^{n}a_{ki}a_{lj}z_{,y_{l}y_{k}},\text{ }1\leq j\leq n.
\end{equation*}%
Then the Hessian of $z\left( x\right) $ follows%
\begin{equation}
Hess\left( z\left( x\right) \right) =%
\begin{bmatrix}
\sum_{k,l=1}^{n}a_{k1}a_{l1}z_{,y_{l}y_{k}} & 
\sum_{k,l=1}^{n}a_{k1}a_{l2}z_{,y_{l}y_{k}} & ... & 
\sum_{k,l=1}^{n}a_{k1}a_{ln}z_{,y_{l}y_{k}} \\ 
\sum_{k,l=1}^{n}a_{k2}a_{l1}z_{,y_{l}y_{k}} & 
\sum_{k,l=1}^{n}a_{k2}a_{l2}z_{,y_{l}y_{k}} & ... & 
\sum_{k,l=1}^{n}a_{k2}a_{ln}z_{,y_{l}y_{k}} \\ 
\vdots & \vdots & \vdots & \vdots \\ 
\sum_{k,l=1}^{n}a_{kn}a_{l1}z_{,y_{l}y_{k}} & 
\sum_{k,l=1}^{n}a_{kn}a_{l2}z_{,y_{l}y_{k}} & ... & 
\sum_{k,l=1}^{n}a_{kn}a_{ln}z_{,y_{l}y_{k}}%
\end{bmatrix}%
.  \tag{3.3}
\end{equation}%
By considering matrix multiplication in $\left( 3.3\right) $ we deduce that 
\begin{equation}
Hess\left( z\left( x\right) \right) =A^{T}\cdot Hess\left( z\left( y\right)
\right) \cdot A,  \tag{3.4}
\end{equation}%
where $A^{T}$ denotes the transpose of $A.$ Thus by $\left( 3.4\right) $ we
obtain (3.2).
\end{proof}

If $\det \left( A\right) \neq 0,$ Lemma 3.1 immediately implies the
following trivial result

\begin{corollary}
A graph of a given smooth real-valued function is flat if and only if so is
its affine graph in $\mathbb{R}^{n+1}$.
\end{corollary}

In particular, the affine graph of (1.1), so-called \textit{affine
translation hypersurface}, has the form 
\begin{equation}
z\left( x_{1},x_{2},...,x_{n}\right) =f_{1}\left( y_{1}\right) +f_{2}\left(
y_{2}\right) +...+f_{n}\left( y_{n}\right) ,\text{ }z=x_{n+1},  \tag{3.5}
\end{equation}%
where $f_{1},f_{2},...,f_{n}$ are arbitrary nonzero smooth functions and $%
\left( y_{1},y_{2},...,y_{n}\right) $ is affine parameter coordinates given
by $\left( 3.1\right) .$ Remark that such a hypersurface reduces to the
standard translation hypersurface, if $A$ is an orthogonal matrix.

Denote $A^{-1}=\left( a^{ij}\right) $ the inverse matrix of $A=\left(
a_{ij}\right) .$ Then, by a change of parameter, the affine translation
hypersurface $M^{n}$ has a parameterization%
\begin{equation}
\left. 
\begin{array}{l}
r\left( y_{1},y_{2},..,y_{n}\right) =\left(
\sum_{i=1}^{n}a^{1i}y_{i},\sum_{i=1}^{n}a^{2i}y_{i},...,\sum_{i=1}^{n}f_{i}%
\left( y_{i}\right) \right)  \\ 
=\underset{\alpha _{1}}{\underbrace{\left(
a^{11}y_{1},a^{21}y_{1},...,f_{1}\left( y_{1}\right) \right) }}+\underset{%
\alpha _{2}}{\underbrace{\left( a^{12}y_{2},a^{22}y_{2},...,f_{2}\left(
y_{2}\right) \right) }}+...+ \\ 
+\underset{\alpha _{n}}{\underbrace{\left(
a^{1n}y_{n},a^{2n}y_{n},...,f_{n}\left( y_{n}\right) \right) }}.%
\end{array}%
\right.   \tag{3.6}
\end{equation}%
Since $A$ is non-orthogonal, so is $A^{-1}$ and this yields that the row and
column vectors of $A^{-1}$ form a non-orthogonal system. Thereby, the
generating curves $\alpha _{1},\alpha _{2},...,\alpha _{n}$ lie in
non-orthogonal planes.
\subsection{Proof of Theorem 1.1.}

We purpose to describe the affine translation hypersurfaces in $\mathbb{%
R}^{n+1}$ with CGKC. For this we need to fix some notations to use in
remaining part:%
\begin{equation}
f_{k}^{\prime }=\frac{df_{k}}{dy_{k}},\text{ }f_{k}^{\prime \prime }=\frac{%
d^{2}f_{k}}{dy_{k}^{2}},\text{ }k=1,2,...,n,  \tag{3.7}
\end{equation}%
and 
\begin{equation}
z_{,x_{i}}=\sum_{k=1}^{n}a_{ki}f_{k}^{\prime },\text{ }z_{,x_{i}x_{j}}=%
\sum_{k=1}^{n}a_{ki}a_{kj}f_{k}^{\prime \prime },\text{ }i,j=1,2,...,n. 
\tag{3.8}
\end{equation}%
By (3.7), the Hessian of $z(y)$ turns to 
\begin{equation}
Hess\left( z\left( y\right) \right) =%
\begin{bmatrix}
f_{1}^{\prime \prime } & 0 & ... & 0 \\ 
0 & f_{2}^{\prime \prime } & ... & 0 \\ 
\vdots & \vdots & \vdots & \vdots \\ 
0 & 0 & ... & f_{n}^{\prime \prime }%
\end{bmatrix}%
.  \tag{3.9}
\end{equation}%
Substituting (3.9) into (3.2) leads to 
\begin{equation}
\det \left[ Hess\left( z(x)\right) \right] =\det \left[ A\right]
^{2}f_{1}^{\prime \prime }f_{2}^{\prime \prime }...f_{n}^{\prime \prime }, 
\tag{3.10}
\end{equation}%
where $x=\left( x_{1},x_{2},...,x_{n}\right) $.

Now we assume that the affine translation hypersurface $M^{n}$ in $\mathbb{R}%
^{n+1}$ has $K=K_{0}=const.$ Then (2.1), (3.7) and $\left( 3.10\right) $
imply that%
\begin{equation}
K_{0}=\frac{\det \left( A\right) ^{2}\left( f_{1}^{\prime \prime
}f_{2}^{\prime \prime }...f_{n}^{\prime \prime }\right) }{\left(
1+\sum_{i=1}^{n}\left( \sum_{j=1}^{n}a_{ji}f_{j}^{\prime }\right)
^{2}\right) ^{\frac{n+2}{2}}}.  \tag{3.11}
\end{equation}

\begin{enumerate}
\item[\textbf{Case 1}] If $K_{0}=0$ in $\left( 3.11\right) ,$ then at least
one of $f_{1},f_{2},...,f_{n}$ is a linear function with respect to the
variables $y_{1},y_{2},...,y_{n}$, respectively. Without lose of generality,
we may assume that $f_{1}\left( y_{1}\right) =cy_{1}+d,$ $c,d\in \mathbb{R}$%
. Considering this one into (3.6), we conclude%
\begin{equation*}
r\left( y_{1},y_{2},..,y_{n}\right) =y_{1}\left( a^{11},a^{21},...,c\right)
+\left(
\sum_{i=2}^{n}a^{1i}y_{i},\sum_{i=2}^{n}a^{2i}y_{i},...,d+%
\sum_{i=2}^{n}f_{i}\left( y_{i}\right) \right) ,
\end{equation*}%
which implies that $M^{n}$ turns to a cylinder.

\item[\textbf{Case 2}] Otherwise, i.e. $K_{0}\neq 0,$ the functions $%
f_{1},f_{2},...,f_{n}$ have to be non-linear. Put $W:=1+\sum_{i=1}^{n}\left(
\sum_{j=1}^{n}a_{ji}f_{j}^{\prime }\right) ^{2}.$ Taking partial derivative
of $\left( 3.11\right) $ with respect to $y_{p},$ $p=1,2,...,n,$ gives%
\begin{equation}
\left( f_{1}^{\prime \prime }f_{2}^{\prime \prime }...f_{p}^{\prime \prime
\prime }...f_{n}^{\prime \prime }\right) W=\left( n+2\right) \left(
f_{1}^{\prime \prime }f_{2}^{\prime \prime }...\left( f_{p}^{\prime \prime
}\right) ^{2}...f_{n}^{\prime \prime }\right) \left(
\sum_{i,j=1}^{n}a_{pi}a_{ji}f_{j}^{\prime }\right) .  \tag{3.12}
\end{equation}%
Since $f_{1}^{\prime \prime }f_{2}^{\prime \prime }...f_{n}^{\prime \prime
}\neq 0,$ $\left( 3.12\right) $ can be rewritten as%
\begin{equation}
\frac{f_{p}^{\prime \prime \prime }}{\left( n+2\right) \left( f_{p}^{\prime
\prime }\right) ^{2}}=\frac{\sum_{i,j=1}^{n}a_{pi}a_{ji}f_{j}^{\prime }}{W}.
\tag{3.13}
\end{equation}%
The partial derivative of (3.13) with respect to $y_{q}$, $p\neq
q=1,2,...,n, $ gives 
\begin{equation}
W\sum_{i,j=1}^{n}a_{pi}a_{qi}-2\left(
\sum_{i,j=1}^{n}a_{qi}a_{ji}f_{j}^{\prime }\right) \left(
\sum_{i,j=1}^{n}a_{pi}a_{ji}f_{j}^{\prime }\right) =0.  \tag{3.14}
\end{equation}%
After twice taking the partial derivative of $\left( 3.13\right) $ with
respect to $y_{q}$ yields%
\begin{equation}
\sum_{i,j=1}^{n}a_{pi}a_{qi}=0.  \tag{3.15}
\end{equation}%
Substituting $\left( 3.15\right) $ into $\left( 3.14\right) $ leads to
either 
\begin{equation}
\sum_{i,j=1}^{n}a_{qi}a_{ji}f_{j}^{\prime }=0\text{ or }%
\sum_{i,j=1}^{n}a_{pi}a_{ji}f_{j}^{\prime }=0.  \tag{3.16}
\end{equation}%
Taking partial derivative in the second equality of $\left( 3.16\right) $
with respect to $y_{p}$ gives%
\begin{equation*}
f_{p}^{\prime \prime }\sum_{i=1}^{n}\left( a_{pi}\right) ^{2}=0
\end{equation*}%
which implies $a_{p1}=a_{p2}=...=a_{pn}=0$. This is a contradiction since $%
\det (A)\neq 0$, which completes the proof.
\end{enumerate}

\section{Further applications}

Before introducing the affine translation hypersurfaces in $\mathbb{I}^{n+1}$, let us
reconsider the notion of translation hypersurface in $\mathbb{I}^{n+1}.$ By means of the
isotropic motions given by (2.3), a \textit{translation hypersurface} in $\mathbb{I}^{n+1}$ generated by translating the curves lying in orthogonal isotropic planes is the graph of the form $\left(1.1\right). $ Such hypersurfaces in $\mathbb{I}^{n+1}$ with constant relative curvature (CRC) and constant isotropic mean curvature (CIMC) were provided in \cite {1}.

Therefore, as similar to Euclidean case, we can state that an \textit{affine
translation hypersurface} in $\mathbb{I}^{n+1}$ is the graph of a function
given via $\left( 3.1\right) $ and $\left( 3.5\right) .$ Point out that the generating curves for this one lie in non-orthogonal isotropic planes. So, by having in mind that the generating curves may also lie non-isotropic planes, the problems given in the Introduction can be also considered in the isotropic spaces. 

By (2.13) and (3.9), for an affine translation hypersurface with CRC $K_{0}$ in $\mathbb{I}^{n+1}$, we get%
\begin{equation}
K_{0}=\det \left( A \right) ^{2}f_{1}^{\prime \prime }f_{2}^{\prime \prime
}...f_{n}^{\prime \prime },  \tag{4.1}
\end{equation}%
where $f_{i}^{\prime \prime }=\frac{d^{2}f_{i}}{dy_{i}^{2}}$\ and $%
(y_{1},y_{2},...,y_{n})$ the affine parameter coordinates given by (3.1).
Hence (4.1) immediately implies that $K_{0}$ vanishes when at least one $%
f_{1},f_{2},...,f_{n}$ is a linear function with respect to the variables $%
y_{1},y_{2},...,y_{n}$, respectively. Suppose that $K_{0}\neq 0.$ Taking
partial derivative of $\left( 4.1\right) $ with respect to $y_{p}$ leads to%
\begin{equation*}
f_{1}^{\prime \prime }f_{2}^{\prime \prime }...f_{p}^{\prime \prime \prime
}...f_{n}^{\prime \prime }=0,
\end{equation*}%
namely%
\begin{equation*}
f_{p}\left( y_{p}\right) =c_{p}y_{p}^{2}+d_{p}y_{p}+e_{p},\text{ }p=1,2,...,n
\end{equation*}%
for some constants $c_{p},d_{p},e_{p}\in \mathbb{R}$, $c_{p}\neq 0$ and $%
c_{1}c_{2}...c_{n}=\frac{K_{0}}{2^{n}\det \left( A \right) ^{2}}.$
Accordingly the following result can be expessed:

\begin{theorem}
Let $M^{n}$ be an affine translation hypersurface in $\mathbb{I}^{n+1}$ with $K_{0}.$ Then, it is
either congruent to a cylinder $\left(
K_{0}=0\right) $ or given by $\left( K_{0}\neq 0\right) $%
\begin{equation*}
\left\{ 
\begin{array}{l}
z\left( x_{1},x_{2},...,x_{n}\right)
=\sum_{i=1}^{n}c_{i}y_{i}^{2}+d_{i}y_{i}+e_{i}, \\ 
c_{i},d_{i},e_{i}\in \mathbb{R},\text{ } c_{i}\neq 0,\text{ }c_{1}c_{2}...c_{n}=%
\frac{K_{0}}{2^{n}\det \left( A \right) ^{2}},\text{ }i=1,2,...,n,%
\end{array}%
\right.
\end{equation*}%
where $\left( y_{1},y_{2},...,y_{n}\right) $ is the affine parameter
coordinates given by $\left( 3.1\right) .$
\end{theorem}

Next we assume that an affine translation hypersurface $M^{n}$ in $\mathbb{I}%
^{n+1}$ has CIMC $H_{0}$. Hence we have from
(2.12) and (3.7) that%
\begin{equation}
nH_{0}=\sum_{i,j=1}^{n}a_{ij}^{2}f_{i}^{\prime \prime }.  \tag{4.2}
\end{equation}%
Taking partial derivative of $\left( 4.2\right) $ with respect to $y_{p},$ $%
p=1,2,...,n,$ gives 
\begin{equation*}
\left(\sum_{i=1}^{n}a_{pi}^{2}\right)f_{p}^{\prime \prime \prime}=0
\end{equation*}
or 
\begin{equation*}
f_{p}\left( y_{p}\right) =\frac{c_{p}}{2\sum_{i=1}^{n}a_{pi}^{2}}%
y_{p}^{2}+d_{p}y_{p}+e_{p}
\end{equation*}%
for some constants $c_{p},d_{p},e_{p}$ such that $%
\sum_{i=1}^{n}c_{i}=nH_{0}. $

Therefore we can present the following result.

\begin{theorem}
Let $M^{n}$ be an affine translation hypersurface in $\mathbb{I}^{n+1}$ with CIMC $H_{0}.$ Then, it is given in explicit form%
\begin{equation*}
\left\{ 
\begin{array}{l}
z\left( x_{1},x_{2},...,x_{n}\right) =\sum_{i=1}^{n}\left(\frac{c_{i}/2}{%
\sum_{j=1}^{n}a_{ij}^{2}}\right)y_{i}^{2}+d_{i}y_{i}+e_{i}, \\ 
\sum_{i=1}^{n}c_{i}=nH_{0},c_{i},d_{i},e_{i}\in \mathbb{R},%
\end{array}%
\right.
\end{equation*}%
where $\left( y_{1},y_{2},...,y_{n}\right) $ is the affine parameter
coordinates given by $\left( 3.2\right) .$ In particular, $M^{n}$ is isotropic minimal provided $\sum_{i=1}^{n}c_{i}=0.$
\end{theorem}

Finally we aim to observe the affine translation hypersurface $M^{n}$ in $%
\mathbb{I}^{n+1}$ whose the coordinate functions are eigenfunctions of the
Laplacian, i.e., that satisfies the condition%
\begin{equation}
\bigtriangleup r_{k}=\lambda _{k}r_{k},\text{ }\lambda _{k}\in \mathbb{R},%
\text{ }k=1,2,...,n+1,  \tag{4.3}
\end{equation}%
where $r_{k}$ is the coordinate function of the position vector of an
arbitray point on $M^{n}$ and $\bigtriangleup $ the Laplace operator of $%
M^{n}$ with respect to the induced metric from $\mathbb{I}^{n+1}$.

In the particular case $\lambda _{1}=\lambda _{2}=...=\lambda _{n+1}=\lambda
,$ the condition $\left( 4.3\right) $ was firstly treated to Riemannian
submanifolds by Tahakashi \cite{29}. Then Garay \cite{10} generalized this
condition as follows:%
\begin{equation*}
\bigtriangleup r=Ar, \text{ } A \in \mathbb{R}_{n+1}^{n+1}.
\end{equation*}%
One is also related to the notion of \textit{submanifolds of finite type} 
conjectured by Chen (see \cite{4,7}).

An affine translation hypersurface $M^{n}$ in $\mathbb{I}^{n+1}$ is of
the form%
\begin{equation*}
r\left( x_{1},x_{2},...,x_{n}\right) =\left(
x_{1},x_{2},...,x_{n},f_{1}\left( y_{1}\right) +f_{2}\left( y_{2}\right)
+...+f_{n}\left( y_{n}\right) \right) ,
\end{equation*}%
where $\left( y_{1},y_{2},...,y_{n}\right) $ is the affine parameter
coordinates given by $\left( 3.1\right) .$ Let us put 
\begin{equation}
r_{1}=x_{1},\text{ }r_{2}=x_{2},...,\text{ }r_{n}=x_{n}  \tag{4.4}
\end{equation}
and%
\begin{equation}
r_{n+1}=f_{1}\left( y_{1}\right) +f_{2}\left( y_{2}\right) +...+f_{n}\left(
y_{n}\right) .  \tag{4.5}
\end{equation}%
From (2.6), $\left( 4.4\right) $ and $\left( 4.5\right) ,$ we conclude that%
\begin{equation}
\bigtriangleup r_{1}=\bigtriangleup r_{2}=...=\bigtriangleup r_{n}=0\text{
and }\bigtriangleup r_{n+1}=\sum_{i,j=1}^{n}a_{ij}^{2}f_{i}^{\prime \prime }.
\tag{4.6}
\end{equation}%
Now suppose that $M^{n}$ holds $\left( 4.3\right) .$ Then $\left( 4.6\right)$
implies $\lambda _{1}=\lambda _{2}=...=\lambda _{n}=0$ and the following
system of ordinary differential equations: 
\begin{equation}
\sum_{i,j=1}^{n}a_{ij}^{2}f_{i}^{\prime \prime }=\lambda
\sum_{i=1}^{n}f_{i}, \text{ }\lambda _{n+1}=\lambda.  \tag{4.7}
\end{equation}%
In the case $\lambda =0,$ $M^{n}$ becomes isotropic minimal stated already
via Theorem 4.2. Hence it is meaningful to assume $\lambda \neq 0.$ Since $%
f_{1},f_{2},...,f_{n}$ depend on the variables $y_{1},y_{2},...,y_{n}$,
(4.7) turns to 
\begin{equation}
\sum_{j=1}^{n}a_{ij}^{2}f_{i}^{\prime \prime }-\lambda f_{i}=\mu _{i}, 
\tag{4.8}
\end{equation}
where $\mu _{i}$ are some constants such that $\sum_{i=1}^{n}\mu _{i}=0.$ If 
$\lambda >0$ in $\left( 4.8\right) ,$ then by solving it we obtain%
\begin{equation*}
f_{i}\left( y_{i}\right) =c_{i}\exp \left( \sqrt{\frac{\lambda }{%
\sum_{j=1}^{n}a_{ij}^{2}}}y_{i}\right) +d_{i}\exp \left( -\sqrt{\frac{%
\lambda }{\sum_{j=1}^{n}a_{ij}^{2}}}y_{i}\right) -\frac{\mu _{i}}{\lambda },
\end{equation*}%
and if $\lambda <0$%
\begin{equation*}
f_{i}\left( y_{i}\right) =c_{i}\cos \left( \sqrt{\frac{-\lambda }{%
\sum_{j=1}^{n}a_{ij}^{2}}}y_{i}\right) +d_{i}\sin \left( \sqrt{\frac{%
-\lambda }{\sum_{j=1}^{n}a_{ij}^{2}}}y_{i}\right) -\frac{\mu _{i}}{\lambda },
\end{equation*}%
where $c_{i},d_{i}$ are some constants.

Therefore we have proved next result.

\begin{theorem}
Let $M^{n}$ be a non isotropic minimal affine translation hypersurface in $%
\mathbb{I}^{n+1}$ satisfying $\bigtriangleup r_{k}=\lambda _{k}r_{k}$. Then $%
\left( \lambda_{1},\lambda_{2},...,\lambda_{n+1}\right) =\left(
0,0,...,\lambda \neq 0\right) $ and $M^{n}$ is congruent to the graph of the
function either

\begin{equation*}
z\left( x_{1},x_{2},...,x_{n}\right) =\sum_{i=1}^{n}c_{i}\exp \left( \sqrt{%
\frac{\lambda }{\sum_{j=1}^{n}a_{ij}^{2}}}y_{i}\right) +d_{i}\exp \left( -%
\sqrt{\frac{\lambda }{\sum_{j=1}^{n}a_{ij}^{2}}}y_{i}\right)
\end{equation*}%
or%
\begin{equation*}
z\left( x_{1},x_{2},...,x_{n}\right) =\sum_{i=1}^{n}c_{i}\cos \left( \sqrt{%
\frac{-\lambda }{\sum_{j=1}^{n}a_{ij}^{2}}}y_{i}\right) +d_{i}\sin \left( 
\sqrt{\frac{-\lambda }{\sum_{j=1}^{n}a_{ij}^{2}}}y_{i}\right) ,
\end{equation*}%
where $\left( y_{1},y_{2},...,y_{n}\right) $ is the affine parameter
coordinates given by $\left( 3.1\right) $ and $c_{i},d_{i}$ some constants.
\end{theorem}

\section{Acknowledgement}
The results of this study were presented at the Conference Pure and Applied Differential Geometry - PADGE, Leuven-Belgium, 2017.

\end{document}